\documentclass[a4paper,10pt]{article}
\long\def\drop#1{}

\usepackage{a4wide}
\usepackage[english]{babel}
\usepackage{graphicx}
\usepackage{listings}
\usepackage{amsmath,amsthm}
\usepackage{amsfonts}
\usepackage{amssymb}
\usepackage{url}

\newtheorem{Theorem}{Theorem}
\newtheorem{Lemma}{Lemma}

\newtheorem{Corollary}{Corollary}

\newtheorem{Example}{Example}

\title{Semicontinuity of Eigenvalues under Flat Convergence in Euclidean Space}
\author{Jacobus W Portegies}

\begin{document}

\maketitle

\begin{abstract}
Recall that Federer-Fleming defined the notion of flat convergence of
submanifolds of Euclidean space to solve the Plateau problem. 
Here we prove the upper semicontinuity of Neumann eigenvalues of the submanifolds when they converge in the flat sense without losing volume. 
With an additional condition on the boundaries of the submanifolds we prove the Dirichlet eigenvalues are semicontinuous as well. 
We show this additional boundary condition is necessary as well as the
condition that the volumes converge to the volume of the limit submanifold. 
As an application of our theorems we see that the Dirichlet and Neumann eigenvalues of a sequence of surfaces with a common smooth boundary curve approaching the solution to the Plateau problem are upper semicontinuous. 
This work is built upon Fukaya's study of the metric measure convergence of Riemannian manifolds. 
One may recall that Cheeger-Colding proved continuity of the eigenvalues when manifolds with uniform lower Ricci curvature bounds converge in the metric measure sense. 
While they obtain continuity, here, we produce an example demonstrating that continuity is impossible to obtain with our weaker hypothesis.
\end{abstract}

What happens to the eigenvalues of the Laplace operator on manifolds when the manifolds converge? 
In case the manifolds in the sequence have uniformly bounded sectional curvatures and converge in the topology of metric measure spaces, Fukaya \cite{fukaya_collapsing_1987} shows that the eigenvalues of the Laplacian converge to the eigenvalues of a self-adjoint, positive operator on the limit space. 
In the same paper, Fukaya already conjectures that it should be possible to obtain the same result if it is only known that the Ricci curvature is bounded from below. 

In the third of a series of papers, Cheeger and Colding \cite{cheeger_structure_2000} give a positive answer to the conjecture. 
They show how you can define a Laplace operator on the limit space and that the eigenvalues of Laplace operators on the manifolds in the sequence converge.

Sometimes it is useful to consider a topology different from the one of metric measure spaces. For instance, the intrinsic flat convergence introduced by Sormani and Wenger \cite{sormani_intrinsic_2011} has been applied by 
Lee and Sormani \cite{lee_stability_2011} to show a version of stability of the Schoen-Yau Positive Mass Theorem \cite{schoen_proof_1979}.
The idea is that manifolds induce currents on the underlying metric space. 
The \emph{intrinsic flat distance} between those currents is calculated by taking the infimum of the flat distances between the push forwards of the currents under all possible isometries into all possible common metric spaces. 
One of the advantages of considering the induced currents, is that the notion of boundary is retained.

We will look at the related question for flat convergence of currents in Euclidean space, as introduced by Federer and Fleming \cite{federer_normal_1960}. 
We wonder what happens to the eigenvalues of the Laplace operator on the manifolds when the induced currents converge in the flat distance to a limit current. 
As we do not assume any bounds on the curvature of the manifolds, we cannot expect that the eigenvalues will be continuous (c.f. \cite{fukaya_collapsing_1987} and Example \ref{ex:surfacerevexample}). 
As Fukaya \cite{fukaya_collapsing_1987} shows in the case of metric measure convergence, with a suitable definition of the eigenvalues on the limit spaces, they do exhibit \emph{semicontinuity}. 
We obtain the following analogous theorem in case of flat convergence of the induced currents.

\begin{Theorem}
\label{co:semimanifoldneumann}
Let $M_i$ ($i=1,2,\dots$) and $M$ be smooth, oriented, compact Riemannian submanifolds of Euclidean space with (possibly empty) boundary, such that the currents induced by $M_i$ as $i\to\infty$ converge in the flat sense to the current induced by $M$. 
Moreover, assume that $\mathrm{Vol}(M_i) \to \mathrm{Vol}(M)$. 
Then
\begin{equation}
\limsup_{i \to \infty} \lambda_k(M_i) \leq \lambda_k(M),
\end{equation} 
where $\lambda_k$ is the $k$th eigenvalue of the (positive) Laplace(-Beltrami) operator on the manifold with Neumann boundary condition.
\end{Theorem}

The eigenvalues for the Dirichlet problem are also upper semicontinuous, as long as the boundaries satisfy a one-sided Hausdorff convergence, as stated precisely in the following theorem. 
We will show in Example \ref{ex:dirichlet} that we cannot just remove this condition on convergence of the boundaries.

\begin{Theorem}
\label{co:semimanifolddirichlet}
Let $M_i$ ($i=1,2,\dots$) and $M$ be smooth, oriented, compact Riemannian submanifolds of Euclidean space with boundary, such that $\mathrm{Vol}(M_i) \to \mathrm{Vol}(M)$ and the currents induced by $M_i$ as $i \to \infty$ converge in the flat sense to the current induced by $M$. 
Moreover, assume that for every $\epsilon > 0$, the boundaries $\partial M_i$ are eventually contained in the tubular $\epsilon$-neighborhood $(\partial M)_\epsilon$ of $\partial M$.
Then, with $\hat{\lambda}_k$ denoting the $k$th eigenvalue of the Laplace operator with Dirichlet boundary condition,
\begin{equation}
\limsup_{i \to \infty} \hat{\lambda}_k(M_i) \leq \hat{\lambda}_k(M).
\end{equation}

\end{Theorem}

Both theorems follow directly from the more general Theorems \ref{th:currentneumann} and \ref{th:currentdirichlet}. 
These theorems prove semicontinuity of functions $\lambda_k$ defined on currents, that reduce to the eigenvalues of the Laplace operator if the currents are induced by manifolds.
In order to get the semicontinuity, we assume that the currents converge weakly and the total mass of the currents does not drop in the limit.

The following corollary provides an example application.
\begin{Corollary}
\label{co:minimizingsurface}
Suppose $C$ is an oriented closed curve in $\mathbb{R}^3$ that has a unique smooth area-minimizing surface $M$ with $\partial M = C$.
If $M_j$ are submanifolds with the same boundary whose area converge to the area of $M$, and all of the $M_j$ are contained in a ball of radius $R>0$ around the origin, then the limsup of the Dirichlet and Neumann eigenvalues of $M_j$ are less than or equal to the corresponding Dirichlet and Neumann eigenvalues of $M$ respectively.
\end{Corollary}

The existence of a smooth minimal surface follows directly from the theory of currents and the regularity theory for minimal currents. By the Federer-Fleming Compactness Theorem, every subsequence has yet another subsequence that converges in the flat sense. Because we assume uniqueness of the minimal surface, the limit of all these subsequences is the same, and the $M_j$ actually converge in the flat distance to the minimal surface. 

We will recall some properties of currents, and clarify our notation in Section \ref{se:currentproperties}. Section \ref{se:defevfunc} introduces the eigenvalue functions $\lambda_k$ on currents and shows how they correspond to eigenvalues when the currents are induced by smooth manifolds. In Section \ref{se:uppersemicontinuity} we will show the semicontinuity of the functions $\lambda_k$ under weak convergence. Section \ref{se:examples} has an example showing the necessity of the assumption on the convergence of volumes, an example demonstrating that the eigenvalues are in general not continuous under flat convergence, and an example which shows that the assumption of the convergence of the boundaries is necessary in order to get semicontinuity of the eigenvalues for the Dirichlet problem.

In future work, we hope to show the upper semicontinuity of the eigenvalues under intrinsic flat convergence.

\subsection*{Acknowledgments}

I would like to thank Carolyn Gordon for suggesting to Christina Sormani to look at the behavior of eigenvalues under intrinsic flat convergence, and Kenji Fukaya for his suggestion to only prove semicontinuity of the eigenvalues. I thank Cristina Sormani for organizing the reading seminar at CUNY and suggesting that I work on the Euclidean submanifold version of this problem and to investigate the Dirichlet and Neumann cases as well. I thank my doctoral advisor, Fanghua Lin, for teaching me Geometric Measure Theory.

\section{Some preliminary properties of currents}
\label{se:currentproperties}

We would first like to recall some properties of currents in Euclidean space. 
We refer to \cite{lin_geometric_2002} and \cite{federer_geometric_1996} for details.

An $n$-current $T$ on $\mathbb{R}^N$ ($n \leq N$) is a functional on the space $\mathcal{D}_n$ of smooth differential $n$-forms on $\mathbb{R}^N$. 
We will write $\mathcal{E}_n$ for the space of $n$-currents.
To $T\in \mathcal{E}_n$ we associate a set function $\|T\|$, initially defined only on open sets $O$ by 
\begin{equation}
\|T\|(O) = \sup\{ T(\phi) \, | \, 
\phi \in \mathcal{D}_n, \, \| \phi \| \leq 1 , 
\, \mathrm{supp}(\phi) \subset O \}.
\end{equation}
Here, for an $n$-covector $\phi\in \wedge^n \mathbb{R}^N$, $\|\phi\|$ represents the comass,
\begin{equation}
\|\phi\| = \sup\{ \langle \xi, \phi \rangle \, | \, \xi \in \wedge_n \mathbb{R}^N, \xi \text{ simple }, |\xi| \leq 1 \},
\end{equation}
and $|\xi|$ is the Euclidean norm of the simple vector $\xi$.
If $\|T\|(\mathbb{R}^N) < \infty$, we say that $T$ has finite mass. 
We will assume this condition is satisfied throughout the paper.
Then $\|T\|$ extends to a Radon measure on $\mathbb{R}^N$, for which we will use the same notation and which we refer to as the mass of $T$, see for instance \cite[4.1.7]{federer_geometric_1996}.

The mass is lower semicontinuous, in that for any open set $O \subset \mathbb{R}^N$, and $T_i \rightharpoonup T$ weakly,
\begin{equation}
\|T\|(O) \leq \liminf_{i\to \infty} \|T_i\|(O).
\end{equation}

If an $n$-current $T$ has finite mass, it is representable by integration. That means that there exists a weakly $\|T\|$-measurable function $\xi:\mathbb{R}^N \to \wedge_n \mathbb{R}^N$, such that for $\|T\|$-almost every $x \in \mathbb{R}^N$, $\xi(x)$ is a simple vector with $|\xi(x)| = 1$, and for all $\phi \in \mathcal{D}_n$,
\begin{equation}
T(\phi) = \int \langle \phi, \xi \rangle \, d \| T \|.
\end{equation}
If both the mass of a current and the mass of its boundary are finite, the current is called \emph{normal}.

The boundary $\partial T$ of an $n$-current $T$ is an $(n-1)$-current defined by
\begin{equation}
\partial T (\omega) = T ( d \omega), \qquad \text{ for all } \omega \in \mathcal{D}_{n-1}.
\end{equation}

We say an $n$-current is rectifiable if $\|T\|$ is supported on a countably $n$-rectifiable set $A$, and $\|T\|$ is absolutely continuous with respect to $H^n\llcorner A$. In that case there exists a positive function $\theta$ on $A$ such that
\begin{equation}
\label{eq:reprrect}
T(\phi) = \int_A \langle \phi, \xi \rangle \theta \, d H^n.
\end{equation}
The current $T$ is called an integral rectifiable current if $\theta$ takes on integer values $H^n$-almost everywhere.

The flat distance $d_F$ between two $n$-currents $S$ and $T$ is given by
\begin{equation}
d_F( S, T ) := \inf \{ \|U\|(\mathbb{R}^N) + \|V\|(\mathbb{R}^N)
\, | \, S - T = U + \partial V, U \in \mathcal{E}_n, V \in \mathcal{E}_{n+1} \}.
\end{equation}
If currents converge in the flat distance, they converge weakly.

By the set $\mathrm{set}(T)$ of an $n$-current $T$, we mean the set of positive lower density, that is the set where
\begin{equation}
\label{eq:deflowerdensity}
\Theta_{*n}(\|T\|,x) := \liminf_{r \to 0} \frac{\|T\|(B(x,r))}{r^n}
\end{equation}
is positive. The set $\mathrm{set}(T)$ of an integral rectifiable $n$-current $T$ is the minimal possible choice of $A$ in the representation (\ref{eq:reprrect}).

Every $n$-dimensional, orientable submanifold $M$ in Euclidean space with finite volume induces an integer rectifiable $n$-current $[|M|]$ on the space, by
\begin{equation}
[|M|](\omega) = \int_M \omega, \qquad \omega \in \mathcal{D}_n.
\end{equation}
The notion of the boundary of a current and the boundary of a manifold correspond well with each other, in that for a smooth manifold $M$ with boundary $\partial M$,
\begin{equation}
[| \partial M |] = \partial [| M |].
\end{equation}

\section{Definition of eigenvalue functions}
\label{se:defevfunc}

The (positive) Laplace(-Beltrami) operator on a smooth manifold is an unbounded, self-adjoint operator and the spectral theorem applies. 
We wonder what happens to the spectrum when a sequence of manifolds converges. 
The notion of convergence that we are considering in this paper, is the one where the induced currents converge weakly to a limit current. 
The underlying set of the limit current is not necessarily a manifold anymore, but with the right conditions, it will still be rectifiable.

For sets that are just rectifiable, it is not so clear what would be the definition of the Laplacian. 
If we would know additionally that the sets have a doubling condition and some sort of Poincar\'{e} inequality, such a definition is possible (c.f. \cite{cheeger_differentiability_1999,cheeger_structure_2000}).

In this paper, we would like to follow the approach by Fukaya \cite{fukaya_collapsing_1987}, in which the eigenvalues are replaced by functions on the weaker space (in this case, the space of currents), that coincide with the eigenvalues when the elements in the weaker space actually come from manifolds. 
The definition is related to the min-max principle (c.f. \cite{reed_analysis_1978}).

For $k \in \mathbb{N}$, we define the functions $\lambda_k$ on the space of $n$-currents on $\mathbb{R}^N$ with finite, compactly supported  mass as follows: if $S$ is such a current, then

\begin{equation}
\label{eq:minmaxforlk}
\lambda_k(S) = \inf_{\Lambda^k\subset C^\infty(\mathbb{R}^N)} \sup_{f \in \Lambda^k\backslash \{0\}} 
\frac{ \int |\nabla f|^2  \, d \| S \| }
     { \int |f|^2 \, d \| S \| },
\end{equation}
where the infimum is over all $k$-dimensional subspaces $\Lambda^k$ of $C^\infty(\mathbb{R}^N)$  and by convention
\begin{equation}
\label{eq:convention}
\frac{\int | \nabla f |^2 \,d \| S \|}
     {\int |f|^2 \, d \|S\| } 
= \infty
\end{equation}
when $f$ is zero $\|S\|$-almost everywhere.

Fukaya \cite[(8.1)]{fukaya_collapsing_1987} similarly defines functions $\lambda_k$ on pairs $(X,\mu)$ where $\mu$ is a Borel measure on some Riemannian manifold $L$, whose support is contained in the subspace $X \subset L$. 
In fact, his definition does not really depend on $X$ and could instead be seen as a definition of functions $\lambda_k$ on measures on $L$. 
If we take $L = \mathbb{R}^N$, $X=M$ and the measure $\mu = \| S \|$, then our definition corresponds exactly with Fukaya's.

Besides associating a current, one can also associate a metric measure space $(M,\mu_M)$ to a Riemannian manifold $M$: the metric space would be the one induced by $M$, and $\mu_M = \mathrm{Vol}_M / \mathrm{Vol(M)}$, where $\mathrm{Vol}_M$ is the volume element of $M$. The measures $\mu_M$ and $\| [|M|] \|$ differ by a positive factor and therefore $\lambda_k(\mu_M)$ according to Fukaya's definition, and $\lambda_k([|M|])$ as defined above are the same. 

Similarly, we define functions $\hat{\lambda}_k$ on normal integer rectifiable $n$-currents $S$, for which the mass and the mass of the boundary are compactly supported, by
\begin{equation}
\hat{\lambda}_k (S):= 
\inf_{\Lambda^k \subset C^\infty_c(\mathbb{R}^N \backslash \mathrm{set} \, \partial S)} \sup_{f \in \Lambda_k \backslash \{ 0 \} } 
\frac{\int |\nabla f|^2 \, d\mu }
     {\int |f|^2 \, d\mu},
\end{equation}
where $\Lambda^k$ are $k$-dimensional subspaces of $C^\infty_c(\mathbb{R}^N \backslash \mathrm{set} \, \partial S)$. 

The following three Lemmas are all in the same spirit. They say that in certain cases in which the currents are induced by manifolds, the values of the $\lambda_k$ coincide with the eigenvalues of the Laplace operator on the manifolds. The first Lemma was stated but not proved by Fukaya in \cite[Lemma 8.3]{fukaya_collapsing_1987}, since the paper contained a proof of another Lemma that was similar but more difficult. We give a proof in this simpler case.

\begin{Lemma}
\label{le:lkforriemannian}
If $M$ is a compact, smooth Riemannian submanifold of $\mathbb{R}^N$ without boundary, then $\lambda_k([|M|])$ corresponds to the $k$th eigenvalue of the Laplace-Beltrami operator on $M$.
\end{Lemma}

\begin{proof}
For notational convenience, set $S := [| M |]$. 
Let $\gamma_k$ ($k = 1,2,\dots$) denote the eigenvalues of the Laplace-Beltrami operator on $M$. 
By the min-max principle,
\begin{equation}
\label{eq:minmaxforgk}
\gamma_k = \inf_{\Gamma^k \subset C^\infty(M)} \sup_{g \in \Gamma^k\backslash \{0\}} 
\frac{\int |\nabla_M g|^2 \, d\mathrm{Vol}_M }
     {\int |g|^2 \, d\mathrm{Vol}_M },
\end{equation}
where the $\inf$ is over all $k$-dimensional subspaces of $C^\infty(M)$ and $\nabla_M$ is the intrinsic gradient on the manifold. 

Fix $\delta > 0$. By (\ref{eq:minmaxforlk}) we can find a $k$-dimensional subspace $\Lambda^k$ of $C^\infty(\mathbb{R}^N)$ such that 
\begin{equation}
\label{eq:pickLk}
\sup_{f \in \Lambda^k\backslash \{0\}} \frac{\int |\nabla f|^2 \, d \| S \| }
                        {\int |f|^2 \, d \| S \|  }
<
\lambda_k( S ) + \delta.
\end{equation}
Define
\begin{equation}
\Gamma^k := \{ g \in C^\infty(M) \, | \, g = f|_M, f \in \Lambda^k \}.
\end{equation}
Then $\Gamma^k$ is a linear subspace of $C^\infty(M)$. 
It is $k$-dimensional, because of the convention (\ref{eq:convention}). Indeed, if $f_1, \dots, f_k$ form a basis of $\Lambda^k$, then $g_i := f_i |_M$ ($i = 1, \dots, k$) span $\Gamma^k$. Now let $\alpha_1, \dots, \alpha_k \in \mathbb{R}$ be such that
\[
\alpha_1 g_1 + \cdots + \alpha_k g_k \equiv 0.
\]
If we take $f := \alpha_1 f_1 + \cdots + \alpha_k f_k$, we see that (\ref{eq:convention}) and (\ref{eq:pickLk}) imply that $f \equiv 0$. Hence $\alpha_1 = \cdots = \alpha_k = 0$, and the $g_i$ are independent for $i=1,\dots,k$.

Since $M$ is a compact, smooth manifold, there exists an $\epsilon>0$ such that in the tubular neighborhood $M_\epsilon\subset \mathbb{R}^N$ of $M$ of size $\epsilon$, the shortest distance projection $P:M_\epsilon \to M$ onto the manifold is well-defined and smooth. 
On $M$, the map $dP:TM_\epsilon \to TM$ is the orthogonal projection onto the tangent space to $M$ and $\nabla_M (f|_M) = dP (\nabla f)$. We find $|\nabla_M (f|_M)| \leq |\nabla f|$. 
Since also $\mathrm{Vol}_M = \|S\|$,
\begin{equation}
\sup_{g \in \Gamma^k\backslash \{0\}} \frac{\int |\nabla_M g|^2 \, d\mathrm{Vol}_M }
                       {\int |g|^2 \, d\mathrm{Vol}_M }
\leq
\sup_{f \in \Lambda^k\backslash \{0\}} \frac{\int |\nabla f|^2 \, d \| S \| }
                       {\int |f|^2 \, d \| S \| }
< \lambda_k(S) + \delta.
\end{equation}
By (\ref{eq:minmaxforgk}), we find $\gamma_k < \lambda_k(S) + \delta$ and since $\delta > 0$ was arbitrary,
\begin{equation}
\gamma_k \leq \lambda_k(S).
\end{equation}

Now let us prove the other inequality. 
If we write $\phi_k$ for the $k$th eigenfunction of the Laplace-Beltrami operator, and $\tilde{\Gamma}^k := \mathrm{span}\{\phi_1, \dots, \phi_k\}$, it holds that
\begin{equation}
\gamma_k 
= \sup_{g \in \tilde{\Gamma}^k\backslash \{0\}} \frac{ \int |\nabla_M g|^2 \, d \mathrm{Vol}_M }
                              { \int |g|^2 \, d \mathrm{Vol}_M }
= \frac{\int |\nabla_M \phi_k|^2 \, d\mathrm{Vol}_M}
                {\int |\phi_k|^2 \, d\mathrm{Vol}_M}.
\end{equation}
Let $\rho: \mathbb{R} \to [0,\infty)$ be a smooth, even function, nonincreasing on the positive halfline such that
\begin{equation}
\label{eq:defrho}
\rho(x) = 
\begin{cases}
1,  & |x| \leq \tfrac{1}{2}, \\
0,  & |x| \geq 1.
\end{cases}
\end{equation}
Define $\psi_k \in C^\infty(\mathbb{R}^N)$ by 
\begin{equation}
\psi_k(x) =
\begin{cases}
\phi_k( P(x) ) \rho( d(x,M)/\epsilon ),& x \in M_\epsilon,\\
0, & x \in \mathbb{R}^N \backslash M_\epsilon,
\end{cases}
\end{equation}
where $d(x,M)$ denotes the distance from $x$ to $M$. 
This definition implies that on $M$, $\phi_k = \psi_k$ and $\nabla_M \phi_k = \nabla \psi_k$. 
Let $\tilde{\Lambda}^k := \mathrm{span}\{ \psi_1, \dots, \psi_{k}\}$. 
Since the functions $\psi_i$ are independent when restricted to $M$, they are certainly independent as functions on $\mathbb{R}^N$. 
It follows that $\tilde{\Lambda}^k$ is $k$-dimensional.
Consequently,
\begin{equation}
\sup_{f \in \tilde{\Lambda}^k\backslash \{0\}} \frac{\int |\nabla f|^2 d\| S \|}
                        {\int |f|^2 d\| S \|}
\leq \sup_{g \in \tilde{\Gamma}^k\backslash \{0\}} \frac{ \int |\nabla_M g|^2 d \mathrm{Vol}_M }
                              { \int |g|^2 d \mathrm{Vol}_M } = \gamma_k.
\end{equation}
Hence, $\lambda_k( S ) = \gamma_k$.
\end{proof}

A similar statement holds for the Neumann eigenvalues in case the manifold has boundary.

\begin{Lemma}
If $M$ is a smooth submanifold of $\mathbb{R}^N$ with boundary and $\bar{M}$ is compact, then $\lambda_k([|M|])$ equals the $k$th eigenvalue of the Laplace-Beltrami operator with Neumann boundary conditions.
\end{Lemma}

\begin{proof}
The first part of the proof is exactly the same as in Lemma \ref{le:lkforriemannian}, as (c.f. \cite[Chapter 7]{davies_spectral_1996}) also for the $k$th eigenvalue of the Laplace operator with Neumann boundary conditions, $\gamma_k$,

\begin{equation}
\gamma_k = \inf_{\Lambda^k \subset C^\infty(\bar{M})} \sup_{g \in \Lambda^k \backslash \{0\} } 
\frac{\int |\nabla_M g|^2 \, d \mathrm{Vol}_M}
     {\int |g|^2 \, d \mathrm{Vol}_M}.
\end{equation}
The same arguments as in Lemma \ref{le:lkforriemannian} show that $\gamma_k \leq \lambda_k$.

For the second part, $M$ can be embedded into a smooth $n$-dimensional submanifold of $\mathbb{R}^N$ without boundary, $\tilde{M}$. 
By regularity theory, the eigenfunctions $\phi_k$ of the Laplace operator with Neumann boundary conditions can be extended to functions in $C_c^\infty(\tilde{M})$, for which we use the same notation.
There is a tubular neighborhood $U$ of $M$ such that the shortest distance projection $P$ onto $\tilde{M}$ is well-defined and smooth in $U$.
By choosing $\epsilon$ small enough, the function 
\begin{equation}
\psi_k(x) = 
\begin{cases}
  \phi_k(P(x)) \rho( d(x, P) / \epsilon ),&  x \in U,\\
  0, & \text{otherwise},
\end{cases}
\end{equation}
with $\rho$ as in the proof of Lemma \ref{le:lkforriemannian}, is in $C^\infty(\mathbb{R}^N)$. Moreover, on $M$ the function values and the gradient correspond to those of $\phi_k$. Hence $\lambda_k = \gamma_k$.
\end{proof}

Finally, we have an analogous statement for the Dirichlet eigenvalue problem.

\begin{Lemma}
If $M$ is a smooth submanifold of $\mathbb{R}^N$ with smooth boundary and $\bar{M}$ is compact, then $\hat{\lambda}_k( [|M|] )$ equals the $k$th eigenvalue of the Dirichlet eigenvalue problem for the Laplacian on $M$.
\end{Lemma}

\begin{proof}
For the $k$th eigenvalue of $M$ (c.f. \cite[Chapter 6]{davies_spectral_1996}),
\begin{equation}
\gamma_k = \inf_{\Gamma^k \subset C_c^\infty(M) } \sup_{g \in \Gamma^k \backslash \{0\}} 
\frac{\int |\nabla_M g|^2 \, d\mathrm{Vol}_M}
     {\int |g|^2 \, d \mathrm{Vol}_M }.
\end{equation}
We can therefore apply the same arguments as in Lemma \ref{le:lkforriemannian} to show that $\gamma_k \leq \hat{\lambda}_k(S)$. 

For the opposite inequality, let $\delta > 0$ and let $\tilde{\Gamma}^k \subset C^\infty_c(\mathbb{R}^N \backslash \mathrm{set}\,(\partial S))$ such that
\begin{equation}
\sup_{g \in \tilde{\Gamma}^k \backslash \{0\}} \frac{ \int |\nabla_M g|^2 \, d \mathrm{Vol}_M} 
                           { \int |g|^2 \, d \mathrm{Vol}_M  } < \gamma_k + \delta.
\end{equation}
Let $\{g_1, \dots, g_k \}$ be an orthogonal basis of $\tilde{\Gamma}^k$ with respect to the $L^2$ inner product on $M$. 
Every $g_i$ is compactly supported. 
There exists a neighborhood $U$ of the support of $g_i$ in $\mathbb{R}^N$ on which the shortest-distance projection onto $M$ is well defined and smooth.
We can then pick $\epsilon > 0$ so small that definining $f_i$ by
\begin{equation}
f_i(x) := 
\begin{cases}
g_i( P(x) ) \rho( d(x,M) / \epsilon  ),& x \in U, \\
0 & \text{ otherwise },
\end{cases}
\end{equation}
with again $\rho$ as in Lemma \ref{le:lkforriemannian}, 
it is guaranteed that $f_i \in C^\infty(\mathbb{R}^N \backslash \mathrm{set}\, \partial S)$. 
On $M$, $f_i \equiv g_i$ and $\nabla_M g_i = \nabla f_i$.
Thus, if we define $\tilde{\Lambda}^k := \mathrm{span} \{f_1, \dots, f_k \}$, then
\begin{equation}
\hat{\lambda}_k (S) 
\leq \sup_{f \in \tilde{\Lambda}^k \backslash \{0\}} 
\frac{\int |\nabla f|^2 \, d \mathrm{Vol}_M}
     {\int |f|^2 \, d\mathrm{Vol}_M}
= \sup_{g \in \tilde{\Gamma}^k \backslash \{0\}}
\frac{\int |\nabla_M g|^2 \, d \mathrm{Vol}_M}
     {\int |g|^2 \, d\mathrm{Vol}_M} < \gamma_k + \delta.
\end{equation}
Because $\delta>0$ was arbitrary, we find $\lambda_k = \gamma_k$.
\end{proof}

\section{Upper semicontinuity}
\label{se:uppersemicontinuity}

We will now show that under a certain condition on the currents that effectively prohibits cancellation of mass in the limit, the eigenvalues of the Laplace operator are upper semicontinuous. 

\begin{Theorem}
\label{th:currentneumann}
Let $T_i$ ($i=1,2, \dots$) and $T$ be compactly supported $n$-currents in $\mathbb{R}^N$ with finite mass and let $T_i \rightharpoonup T$ weakly and $\|T_i\|(\mathbb{R}^N) \to \|T\|(\mathbb{R}^N)$.
Then the functions $\lambda_k$ as defined in (\ref{eq:minmaxforlk}) are upper semicontinuous. 
That is, for every $k \in \mathbb{N}$,
\begin{equation}
\limsup_{i \to \infty} \lambda_k(T_i) \leq \lambda_k( T ).
\end{equation}
\end{Theorem}

We will use the following lemma, which shows weak convergence for the mass measures associated to the currents. 

\begin{Lemma}
\label{le:weakconvergence}
Assume that $T_i$ ($i = 1,2, \dots$) and $T$ are $n$-currents on $\mathbb{R}^N$ with $T_i \rightharpoonup T$ weakly and $\|T_i\|(\mathbb{R}^N) \to \|T\|(\mathbb{R}^N)$. Then for all $\phi \in C_c(\mathbb{R}^N)$,
\begin{equation}
\lim_{i \to \infty} \int \phi \, d \|T_i\| = \int \phi \, d \| T \|.
\end{equation}
\end{Lemma}

The statement of the lemma is slightly nonstandard, but the conclusion follows directly from the lower semicontinuity of the mass and the convergence of the total mass (c.f. \cite[Section 1.9]{evans_measure_1991}). For completeness of exposition we include a different proof at the end of this section.

With the lemma at hand, we can prove Theorem \ref{th:currentneumann}.

\begin{proof}[Proof of Theorem \ref{th:currentneumann}]
Let $\delta > 0$. By the definition of $\lambda_k$ in (\ref{eq:minmaxforlk}), we can fix a subspace $\Lambda^k \subset C^\infty(\mathbb{R}^N)$ such that
\begin{equation}
\label{eq:estimateforLk}
\sup_{f \in \Lambda^k \backslash \{ 0 \} } \frac{\int |\nabla f|^2 \, d \|T\|}{\int |f|^2 \, d\|T\|} \leq \lambda_k(T) + \delta.
\end{equation}
In fact, we may assume that $\Lambda^k \subset C_c^\infty(\mathbb{R}^N)$, since $T$ is assumed to have compact support.
By Lemma \ref{le:weakconvergence}, we have for $i$ large enough
\begin{equation}
\sup_{f \in \Lambda^k \backslash \{ 0 \}} \frac{\int |\nabla f|^2 \, d\|T_i\|}{ \int|f|^2 \, d\|T_i\| }
\leq \sup_{f \in \Lambda^k \backslash \{0\} } \frac{\int |\nabla f|^2 \, d \|T\|}{\int |f|^2 \, d\|T\|} + \delta.
\end{equation}
Therefore, for $i$ large enough,
\begin{equation}
\lambda_k(T_i) \leq \lambda_k(T) + 2 \delta,
\end{equation}
and
\begin{equation}
\limsup_{i \to \infty} \lambda_k(T_i) \leq \lambda_k(T) + 2 \delta.
\end{equation}
The theorem follows by taking $\delta \downarrow 0$.
\end{proof}

\begin{Theorem}
\label{th:currentdirichlet}
Let $T_i$ ($i = 1, 2 , \dots$) and $T$ be normal $n$-currents in $\mathbb{R}^N$. 
As in the previous theorem, assume that $T_i \rightharpoonup T$ weakly and $\|T_i\|(\mathbb{R}^N) \to \|T\|(\mathbb{R}^N)$. 
Additionally, we assume that for every $\epsilon > 0$, eventually
$\mathrm{set}\,(\partial T_i) \subset (\mathrm{set}\,(\partial T))_\epsilon$.
Then
\begin{equation}
\limsup_{i \to \infty} \hat{\lambda}_k(T_i) \leq \hat{\lambda}_k( T ).
\end{equation}
\end{Theorem}

\begin{proof}
We note that we can choose $\Lambda^k \subset C^\infty_c(\mathbb{R}^N \backslash \mathrm{set}\, \partial T)$ such that 
\begin{equation}
\sup_{f \in \Lambda^k \backslash\{0\} } \frac{\int |\nabla f|^2 \, d\| T \|}
                        {\int |f|^2 \, d \|T\|} 
\leq \hat{\lambda}_k(T) + \delta.
\end{equation}
By the hypothesis and the fact that the collection of $f \in \Lambda_k$ is compact to a scaling factor, there is an $\epsilon>0$ such that for all $f \in \Lambda^k$, 
\begin{equation}
(\mathrm{set}\, \partial T)_\epsilon \cap \mathrm{supp} f = \emptyset.
\end{equation}
Because of our assumption on $\mathrm{set}\, \partial T_i$, for $i$ large enough 
\begin{equation}
\mathrm{set}\, \partial T_i \subset (\mathrm{set} \, \partial T)_{\epsilon / 2},
\end{equation}
so $\Lambda^k$ is also a $k$-dimensional subspace of $C^\infty_c(\mathbb{R}^N \backslash \mathrm{set}\,(\partial T_i))$ for $i$ large enough. The remainder of the proof is the same as above.
\end{proof}

Since the functions $\lambda_k$ correspond to the eigenvalues of the Laplace operator when the currents are induced by smooth manifolds, and flat convergence implies weak convergence, we immediately obtain Theorems \ref{co:semimanifoldneumann} and \ref{co:semimanifolddirichlet} in the introduction.

We will conclude this section with a proof of Lemma \ref{le:weakconvergence}. We prove the lemma by decomposing $\mathbb{R}^N$ into some well-chosen collection of small-enough cubes $Q$ so that actually $\lim_{i\to\infty} \|T_i\|(Q) = \|T\|(Q)$.

\begin{proof}[Proof of Lemma \ref{le:weakconvergence}]
Let $\epsilon > 0$. Since $\phi$ is uniformly continuous, we can pick $\delta>0$ so small that if $|x - y | < 2\sqrt{N} \delta$, then $|\phi(x) - \phi(y)|< \epsilon$. 

We pick $a \in [0,\delta)^N$ such that $\|T\|(G) = 0$, where
\begin{equation}
G := \{ x \in \mathbb{R}^N \, | \, x_m \in a_m + \delta \mathbb{Z} \text{ for some } m = 1, \dots, N \}.
\end{equation}
We can find such an $a$ because for every $m= 1, \dots, N$, there are only countably many $t\in \mathbb{R}$ such that
\begin{equation}
\|T \| ( \{x \in \mathbb{R}^N \, | \, x_m = t \} ) > 0.
\end{equation}
For $b\in \mathbb{Z}^N$ let $Q^{b}$ denote the cube
\begin{equation}
Q^b := \{ x \in \mathbb{R}^N \, | \, \delta b_1 < x_1 - a_1 < \delta (b_1 + 1) , 
\dots, \delta b_N < x_N - a_N < \delta (b_N + 1) \}.
\end{equation}
Set 
\begin{equation}
O := \bigcup_{b \in \mathbb{Z}^N} Q^{b} = \mathbb{R}^N \backslash G.
\end{equation}
Since $O$ is open, by the lower-semicontinuity of the mass,
\begin{equation}
\liminf_{i \to \infty} \|T_i\|(O) \geq \|T\|(O).
\end{equation}
Our choice of $a$ guarantees that $\|T\|(G) = 0$. 
By our assumption of convergence of the total mass of the space,
\begin{equation}
\limsup_{i \to \infty} \|T_i \| (O) \leq \lim_{i\to\infty} \| T_i \|(\mathbb{R}^N) 
= \|T\|(\mathbb{R}^N) = \|T\|(O).
\end{equation}
So $\lim_{i \to \infty} \|T_i\|(O) = \|T\|(O)$ and $\lim_{i \to \infty} \|T_i\|(G) = 0$. It immediately follows that  
\begin{equation}
\lim_{i \to \infty} \|T_i\|(Q^b) = \|T\|(Q^b),
\end{equation}
for any $b\in \mathbb{Z}^N$.

Hence,
\begin{equation}
\begin{split}
\int \phi \, d \|T_i\| - \int \phi \, d\| T\| &= \sum_{b \in \mathbb{Z}^N} 
\left( \int_{Q^{b}} \phi \, d \|T_i\| -  \int_{Q^b} \phi \, d \|T\|\right) + \int_G \phi \, d\|T_i\|\\
&= \sum_{b\in \mathbb{Z}^N} \left( \phi(x^b) \|T_i\|(Q^{b}) - \phi(y^b) \|T\|(Q^{b}) \right) + \int_G \phi \, d\|T_i\|
\end{split}
\end{equation}
by the mean value theorem, for some $x^{b}$ and $y^{b}$ in $Q^{b}$. Consequently,
\begin{equation}
\begin{split}
\left| \int \phi \, d \|T_i\| - \int \phi \, d\| T\| \right| 
&\leq \sum_{b\in \mathbb{Z}^N} \left|\phi(x^{b})- \phi(y^{b}) \right| \|T_i\|(Q^{b}) \\
&\qquad + \sum_{b \in \mathbb{Z}^N} |\phi(y^{b})| \left| \|T_i\|(Q^{b}) - \|T\|(Q^{b}) \right|\\
& \qquad + \sup |\phi| \|T_i\|(G) \\
& \leq \epsilon \| T_i \| (O) + \sum_{b\in \mathbb{Z}^N} \sup_{y \in Q^b} |\phi(y)| \left| \|T_i\|(Q^{b}) - \|T\|(Q^{b}) \right| + \sup |\phi| \|T_i\|(G).
\end{split}
\end{equation}
Since $\phi$ is compactly supported, only finitely many terms contribute to the sum. 
Hence,
\begin{equation}
\limsup_{i \to \infty} \left| \int \phi \, d \|T_i\| - \int \phi \, d\| T\| \right| \leq \epsilon \|T\|(O).
\end{equation}
Since $\epsilon$ was arbitrary, the lemma is proved.
\end{proof}

\section{Examples}
\label{se:examples}

The following example shows that if we lift the condition that $\|T_i\|(\mathbb{R}^N) \to \|T\|(\mathbb{R}^N)$, the lower semicontinuity fails in general.

\begin{Example}
\label{ex:nototalmassconvergence}
Consider first for $\epsilon > 0$ the functions $x_\epsilon:[0,4] \to \mathbb{R}^2$ given by

\begin{equation}
\label{eq:defbarxeps}
\bar{x}_\epsilon(t) = 
\begin{cases}
\left( \cos\frac{2 \pi (t+\epsilon)}{1 + 2 \epsilon} - 3
       , \sin \frac{2 \pi (t+\epsilon)}{1 + 2\epsilon}  \right), 
& t \in [0,1], \\
\left( (2-t) \left(\cos\frac{2\pi (1+\epsilon)}{1 + 2 \epsilon} - 3\right) 
+ (t-1)\left( 3 - \cos\frac{2\pi \epsilon}{1 + 2 \epsilon} \right)
, - \sin \frac{2\pi \epsilon}{1 + 2 \epsilon } \right),
& t \in [1,2], \\
\left( 3 - \cos \frac{2 \pi (t-2 + \epsilon) }{1 + 2 \epsilon},
- \sin \frac{2\pi (t-2 + \epsilon)}{1 + 2 \epsilon} \right), 
& t \in [2,3], \\
\left( (4-t) \left(3 - \cos\frac{2 \pi (1+\epsilon)}{1+2\epsilon}\right) 
+ (t-3) \left( \cos \frac{2\pi \epsilon}{1 + 2\epsilon} - 3\right),
 \sin \frac{2 \pi \epsilon}{1+ 2\epsilon}  \right),
& t \in [3,4].
\end{cases}
\end{equation}
We sketched the image of $\bar{x}_\epsilon$ in Figure \ref{fig:Teps}. Since these functions are not smooth, we consider instead the functions $x_\epsilon = \bar{x}_\epsilon * \rho_{\epsilon^2}$, where $\rho_{\epsilon^2} = \epsilon^{-2} \rho(x / \epsilon^2)$, and $\rho \in C_c^\infty(\mathbb{R})$ is a standard, symmetric mollifier. For fixed $\epsilon > 0$, the functions $x_\epsilon$ parametrize a smooth one-dimensional Riemannian manifold $M_\epsilon$ in $\mathbb{R}^2$, with the orientation induced by the parametrization, and therefore they induce currents $T_\epsilon = [| M_\epsilon |]$. 

\begin{figure}[h]
  \centering
  \includegraphics[width=12cm]{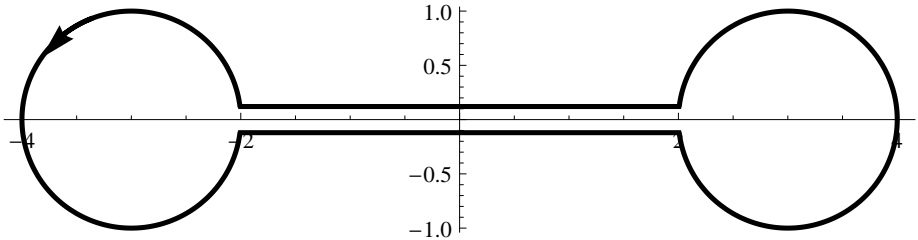}
  \caption{Image of $\bar{x}_\epsilon$, as defined in (\ref{eq:defbarxeps}), with $\epsilon=0.02$}
  \label{fig:Teps}
\end{figure}

As $\epsilon \downarrow 0$, $T_\epsilon \rightharpoonup T$ weakly as currents, where $T$ is the sum of a current induced by a unit circle centered at $(-3,0)$, and a current induced by a unit circle centered at $(3,0)$, both with counterclockwise orientation. 

The eigenvalues of the Laplace operator on a compact one-dimensional manifold parametrized by a smooth Jordan curve are $-(2 \pi \lfloor k/2 \rfloor / L)^2$, for $k = 1, 2$, where $L$ is the total arclength of the Jordan curve.

If a Riemannian manifold $M$ can be divided into two disjoint components $M_1$ and $M_2$, the eigenvalues of $M$ are the eigenvalues of $M_1$ together with the eigenvalues of $M_2$, where eigenvalues are repeated according to multiplicity.

Going back to our example, this implies that
\begin{equation}
\lambda_k(T_\epsilon) 
= \left(\frac{2 \pi \lfloor k / 2 \rfloor}{L_\epsilon}\right)^2, \qquad k = 1, 2, \dots,
\end{equation}
with $L_\epsilon$ the total length of $M_\epsilon$, and
\begin{equation}
\lambda_k(T)
= \left(\frac{2 \pi \lfloor (k+1) / 4 \rfloor}{2\pi}\right)^2, \qquad k = 1, 2, \dots.
\end{equation}

In particular, with Lemma \ref{le:lkforriemannian}, because $L_\epsilon$ is uniformly bounded above,

\begin{equation}
\limsup_{\epsilon \downarrow 0} \lambda_2(T_\epsilon) > 0 = \lambda_2(T),
\end{equation}
showing that in this case, the function $\lambda_2$ is not upper semicontinuous.
\end{Example}

The next example shows that without further assumptions, we cannot expect the eigenvalues of the Laplace operator on the manifolds to behave continuously under flat convergence of the induced currents.

\begin{Example}
\label{ex:surfacerevexample}

If we revolve the parametric curve of $x_\epsilon(t)$, $t\in[1/2, 5/2]$ around the $x$-axis, we obtain two-dimensional manifolds $N_\epsilon$ whose induced currents $S_\epsilon$ provide an example that shows that in general, we cannot expect continuity of the eigenvalues. 
Beale \cite{beale_scattering_1973} and Fukaya \cite{fukaya_collapsing_1987} studied similar examples. 
Their results imply that
\begin{equation}
\lim_{\epsilon \downarrow 0} \lambda_3(S_\epsilon) = \left( \frac{\pi}{4} \right)^2,
\end{equation}
the negative of the first eigenvalue of the Laplace operator on $[-2,2]$ with Dirichlet boundary conditions. 
Note that $S_\epsilon \rightharpoonup S$ weakly, where $S$ is the current induced by two unit spheres, one centered at $(-3,0,0)$ and one centered at $(3,0,0)$. 
Consequently,
\begin{equation}
\lambda_3(S) = 2 > \left(\frac{\pi}{4} \right)^2 = \lim_{\epsilon \downarrow 0} \lambda_3(S_\epsilon).
\end{equation}
The idea is that on $S_\epsilon$, the first eigenvalue will be zero, associated with a constant eigenfunction. When $\epsilon$ gets small, the second eigenvalue will be close to zero as well, and the eigenfunction will be almost constant on the two spheres, with a transition region in between. Finally, the third eigenfunction will be approximately equal to $\cos(\pi x / 4 ) / \sqrt{\epsilon}$ on the tube, and zero on the two spheres.
\end{Example}

We now consider an example in which currents induced by some Riemannian manifolds with boundary converge, but the associated Dirichlet problems do not.

\begin{Example}
\label{ex:dirichlet}
We consider manifolds $M_i \subset \mathbb{R}^2$, 
\begin{equation}
\label{eq:defMi}
M_i := [0,1]^2 - \bigcup_{x \in 2^{-i} \mathbb{Z}^2} B(x, R_0(i)),
\end{equation}
with $R_0:\mathbb{N} \to \mathbb{R}_+$ a function that we will choose later.
In Figure \ref{fig:swisscheese} we sketched $M_2$. 
As we show below, we can choose $R_0(i)$ such that the eigenvalues for the Laplace operator on $M_i$ blow up to $\infty$, but as currents, $[|M_i|] \rightharpoonup [|M|]$ weakly with $M = [0,1]^2$. 
Consequently, in this case there is no upper semicontinuity of the eigenvalues of the Laplace operator.

\begin{figure}[h]
  \centering
  \includegraphics[width=6cm]{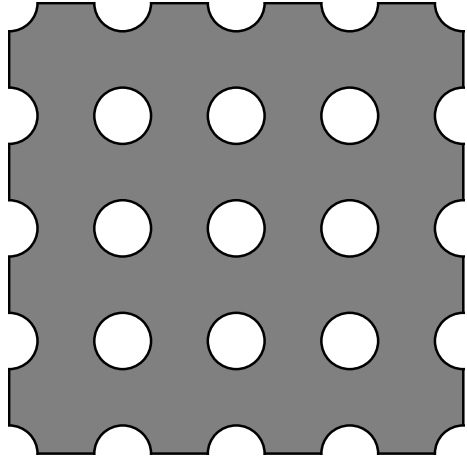}
  \caption{Image of $M_i$ as defined in (\ref{eq:defMi}), with $i=2$.}
  \label{fig:swisscheese}
\end{figure}
\begin{figure}[h]
  \centering
  \includegraphics[width=6cm]{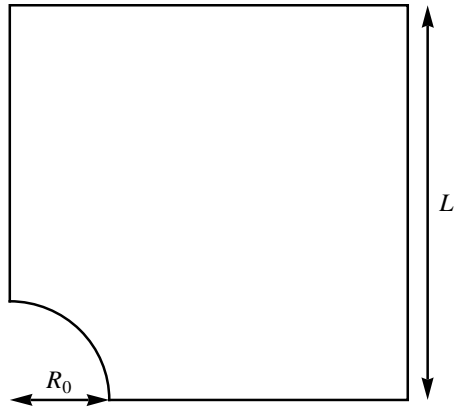}
  \caption{Image of $\Omega_L$ as defined in (\ref{eq:defOmegaL}).}
  \label{fig:poincaredomain}
\end{figure}

We are going to derive a Poincar\'{e} inequality on a certain specific domain $\Omega_L = (0,L)^2 \backslash \overline{B(0,R_0)}$, see Figure \ref{fig:poincaredomain}. 
We choose polar coordinates $(r,\phi)$ in $\mathbb{R}^2$. The function $b(r)$ below is such that 
\begin{equation}
\label{eq:defOmegaL}
\Omega_L = \{ ( r \cos \phi, r \sin \phi ) \, | \, r \in (0,L), \phi \in (b(r),\pi/2 - b(r) ) \},
\end{equation}
that is,
\begin{equation}
b(r) := 
\begin{cases}
0,& 0 < r \leq  L, \\
\arccos\frac{L}{r}, & L < r < \sqrt{2} L.
\end{cases}
\end{equation}
Let $u \in C^\infty(\mathbb{R}^2)$, be such that $u$ vanishes in a neighborhood of the circular arc in $\partial \Omega_L$. Using the fundamental theorem of calculus, we find
\begin{equation}
\begin{split}
\int_{\Omega_L} |u|^2 & =
\int_{R_0}^{\sqrt{2}L} \int_{b(r)}^{\pi/2 - b(r)} |u|^2 \, d\phi \, r \, dr \\
&= \int_{R_0}^{\sqrt{2} L} \int_{b(r)}^{\pi/2-b(r)} \left| 
\int_{R_0}^{r} 
  \frac{\partial u}
       {\partial r} (s, \phi) \,
ds \right|^2 \, d\phi \, r \, dr \\
& \leq \int_{R_0}^{\sqrt{2} L} \int_{b(r)}^{\pi/2-b(r)} ( r - R_0 ) \int_{R_0}^r |\nabla u|^2(s,\phi) 
\frac{s}{s} \, ds \, d\phi \, r \, dr \\
& \leq \frac{1}{R_0} \int_{R_0}^{\sqrt{2} L} \int_{R_0}^r \int_{b(r)}^{\pi/2-b(r)}  |\nabla u|^2(s,\phi) 
s \, d\phi \, ds (r-R_0) r \, dr \\
& \leq \frac{1}{R_0} \int_{R_0}^{\sqrt{2} L} \int_{R_0}^{\sqrt{2} L} \int_{b(s)}^{\pi/2 - b(s)}
|\nabla u|^2 (s,\phi) s \, d\phi \, ds (r - R_0) r \, dr \\
& \leq   \frac{2 \sqrt{2} L^3}{3 R_0 } \int_{\Omega_L} |\nabla u|^2.
\end{split}
\end{equation}

Now, let $f \in C^\infty(\mathbb{R}^2)$ such that $f$ vanishes on $\partial M_i$. When we subdivide $M_i$ in scalings and rotations of $\Omega_L$, apply the above Poincar\'{e} inequality and sum over all contributions, we get
\begin{equation}
\int_{M_i} |f|^2 \leq \frac{ 2^{-3i} }{ R_0(i) } \int_{M_i} |\nabla f|^2.
\end{equation}
Therefore,
\begin{equation}
\lambda_1 \geq \frac{\int_{M_i} |\nabla f|^2 }{ \int_{M_i} |f|^2 } \geq 2^{3i} R_0(i).
\end{equation}
If we choose for instance $R_0(i) = 2^{-5i/2}$, the eigenvalues of $M_i$ diverge to $\infty$, while $M_i$ converges to a solid square in the flat sense.
\end{Example}

\bibliography{ReferencesFunctionsOnCurrents}
\bibliographystyle{plain}

\end{document}